\documentclass[reqno]{amsart}
\usepackage{amsmath,amssymb}
\usepackage{relsize}
\usepackage{graphicx,color}
\usepackage[all]{xy}

\theoremstyle{plain}
\newtheorem{introtheorem}{Theorem}

\newtheorem{theorem}{Theorem}[section]
\newtheorem{proposition}[theorem]{Proposition}
\newtheorem{lemma}[theorem]{Lemma}
\newtheorem{corollary}[theorem]{Corollary}

\theoremstyle{definition}
\newtheorem{definition}[theorem]{Definition}
\newtheorem{example}[theorem]{Example}

\theoremstyle{remark}
\newtheorem{remark}[theorem]{Remark}

\def\E{{\mathcal E}}

\def\D{{\mathcal D}}

\def\G{{\mathcal G}}

\def\T{{\mathbb T}}
\def\L{{\mathbb L}}

\def\G{\mathcal G}

\def\cat0{\mathrm{cat}_0}

\def\ker{\mathrm{ker}}

\def\aut{\mathrm{aut}}

\begin{document}

\title[]
{ Adams-Hilton model and the group of self-homotopy equivalences of a simply connected cw-complex}

\author{Mahmoud Benkhalifa}
\address{Department of Mathematics. Faculty of  Sciences, University of Sharjah. Sharjah, United Arab Emirates}

\email{mbenkhaifa@sharjah.ac.ae}

\begin{abstract} Let $R$ be a principal ideal domain (PID).  For a simply connected  CW-complex  $X$ of dimension $n$,  let $Y$ be  a space obtained by attaching cells of dimension $q$  to  $X$, $q>n$, and  let $A(Y)$ denote an  Adams-Hilton model of $Y$. If  $\mathcal E(A(Y))$ denotes the group of homotopy self-equivalences  of $A(Y)$ and  $\mathcal E_{*}(A(Y))$ its  subgroup of  the  elements inducing the identity on $H_{*}( Y,R)$, then we construct two  short exact sequences:
$$\underset{i}{\oplus}\,H_{q}(\Omega X,R)\rightarrowtail
\mathcal{E}(A(Y))\overset{}{
	\twoheadrightarrow}\Gamma^{q}_{n}\,\,\,\,\,\,\,\,\,\,\,\,,\,\,\,\,\,\,\,\,\,\,\,\,\underset{i}{\oplus}\,H_{q}(\Omega X,R) \rightarrowtail
\E_{*}(A(Y))\overset{}{
	\twoheadrightarrow}\Pi^{q}_{n} $$
where $i=\mathrm{rank} \,H_{q}(Y,X;R)$,  $\Gamma^{q}_{n}$  is a subgroup of $\aut(\mathrm{Hom}_{}(H_{q}( Y,X;R))\times \E(A(X))$   and  $\Pi^{q}_{n}$   is a subgroup of $\mathcal E_{*}(A(X))$.
\end{abstract}
\maketitle
\section{Introduction}
Let $R$ be a PID  and let  $Y$ be  a simply connected  CW-complex. The Adams-Hilton model of  $Y$ is a chain algebra morphism  $$\Theta_{Y}: (\T( V),\partial)\to C_{*}(\Omega Y,R)$$
such that $H_{*}(\Theta_{Y}): H_{*}(\T( V),\partial)\to H_{*}(\Omega Y,R)$
is an isomorphism of graded algebras and such as
$ H_{*}(V,d)\cong H_{*}(Y,R)$ as  graded $R$-modules, where   $d$ denotes  the linear part of the differential $\partial$ induced on the graded module of the indecomposable $V$, where  $C_{*}(\Omega Y,R)$ denotes the complex of the non-degenerate cubic chains equipped with the multiplication induced by the composition of loops and where $(\T(V),\partial)$ is  the  free chain $R$-algebra on the free graded $R$-module $V$.  Let   $A(Y)$ denote the Adams-Hilton  model   of the space  $Y$.

  As is well known,   there is a reasonable concept of “homotopy”
among  chain algebra morphisms (see section 3), analogous in many respects to the topological notion of homotopy. Consequently let $\mathcal E(A(Y))$ denote the group of homotopy self-equivalences  of the chain algebra $A(Y)$.

 By virtues of the Adams-Hilton model, it is worth to mention  that if $\alpha:Y\to Y$ is a homotopy equivalence, then so is $A(\alpha):A(Y)\to A(Y)$. Therefore  there is  a homomorphism of groups $\E(Y)\to \E(A(Y))$ sending $[\alpha]$ to $[A(\alpha)]$, where $\E(Y)$ is the group of homotopy self-equivalences  of $Y$ (see, for example, \cite{Benk7} for more details about this group) .

The  idea  of  inserting  the group $\E(Y)$ in a short exact sequence of groups  of the form $A\rightarrowtail \E(Y)\twoheadrightarrow B$  traces back to the first results on this group in the 1950s.   Barcus-Barrett \cite{BB}   gave  an exact sequence    describing the effect of a single cell attachment $Y = S^n \cup_\alpha S^{q-1}$  on the group $\E(Y)$. This basic result was refined and extended by later authors including P. J. Kahn \cite{Kahn},  Oka-Sawashita-Sugawara \cite{OSS}, Benkhalifa-Smith\cite{Benk7} and Benkhalifa \cite{Benk0,Benk13,Benk3,Benk4}. We refer the reader to \cite{Ru1, Ru} for a comprehensive survey on these   results including various exact sequences.

 The aim of this paper is to study the effect of cell-attachment on the group $\E(A(X))$. More precisely let $X$  be a simply connected   CW-complex of dimension $n$ and let
\begin{equation}\label{020}
Y=X\cup_{\alpha}\Big(\underset{i\in I}{\bigcup} e_{i}^{q}\Big)
\end{equation}
is the space obtained by attaching cells of dimension $q$  to  $X$ by a map $\alpha:\underset{i\in I}{\vee} S_{}^{q-1}\to X$. Let $\mathcal E_{*}(A(Y))$ denote the subgroup of $\mathcal E(A(Y))$  of   the  elements inducing the identity on $H_{*}( Y,R)$. We prove:
\begin{introtheorem}\label{main}
For every $n$ and for every $q> n$, there exist two
short exact sequences of groups
\begin{equation*}
\label{18}
\underset{i}{\oplus}\,H_{q}(\Omega X,R)\rightarrowtail
\mathcal{E}(A(Y))\overset{}{
	\twoheadrightarrow}\Gamma^{q}_{n}\,\,\,\,\,\,\,\,\,\,\,\,,\,\,\,\,\,\,\,\,\,\,\,\,\underset{i}{\oplus}\,H_{q}(\Omega X,R) \rightarrowtail
\E_{*}(A(Y))\overset{}{
	\twoheadrightarrow}\Pi^{q}_{n}
\end{equation*}
where $i=\mathrm{rank} \,H_{q}(Y,X;R)$, where $\Gamma^{q}_{n}$  is a subgroup of $\aut(\mathrm{Hom}_{}(H_{q}( Y,X;R))\times \E(A(X))$   and $\Pi^{q}_{n}$   is a subgroup of $\mathcal E_{*}(A(X))$ ( see definitions 4.1 )
\end{introtheorem}

An analogous problem  was previously studied in \cite{Benk1},  in terms of the Postnikov decomposition of the rational space $Y$  and by the use of the Sullivan model in rational homotopy theory and  it is shown  that
\begin{introtheorem}[\cite{Benk1}, Corollary 3.3]
	There exist two exact short sequences
	$$\mathrm{Hom}\big(\pi_{q}(Y);H^{q}(Y^{[n]})\big) \rightarrowtail
	\mathcal{E}(Y^{[n+1]})\overset{}{
		\twoheadrightarrow}D^{n}_{n-1}$$
	$$\mathrm{Hom}\big(\pi_{q}(Y);H^{q}(Y^{[n]})\big) \rightarrowtail
	\mathcal{E}_{\#}(Y^{[n+1]}))\overset{}{
		\twoheadrightarrow}\mathrm{G}^{n}_{n-1}$$
	where $D^{q}_{n}$ is a subgroup of   $\aut\big(\mathrm{Hom}(\pi_{q}(Y),\Bbb Q)\Big)\times \E(Y^{[n]})$ and where
	$G^{q}_{n}$ is a subgroup of $\mathcal{E}_{\#}(Y^{[n]})$. Here $Y^{[k]}$ denotes  the $k^{\text{th}}$ Postnikov section  of $Y$ and $\mathcal{E}_{\#}(Y^{[n]})$ denotes the subgroup of $\mathcal{E}(Y^{[n]})$ of the elements inducing the identity on the homotopy groups.
\end{introtheorem}

In particular,  let $Y$ be a simply connected CW-complex and let $\Sigma Y$ the suspension of $Y$.   A well-known theorem due to Bott- Sameson \cite{BS} asserts that, under the assumption that the homology $H_{*}(Y;R)$ is a  free graded $R$-module,   the chain algebra    $\big(\T(H_*(Y;R)),0\big)$ (with the trivial differential)  can be considered  as an Adams-Hilton for the space $ \Sigma Y$.   Consequently,  we prove that  the group  $\mathcal{E}(A( \Sigma Y))$  is simply identified with the group $\aut(A( \Sigma Y))$  of the chain algebra automorphisms of $A( \Sigma Y)$ as well as  the subgroup  $\E_{*}(A( \Sigma Y))$ is identified with the subgroup $\aut_*(A( \Sigma Y))$  of the chain algebra automorphisms  inducing the identity on the graded module $H_{*}(Y,R)$. Moreover applying theorem 1, we obtain the following short exact sequences of groups
\begin{introtheorem}
	let $Y$ be a simply connected CW-complex and let $\Sigma Y$ the suspension of $Y$. 	There exist two exact short sequences
$$\mathrm{Hom}\Big(H_{q+1}( Y;R),\T_{q}(H_{<q}(Y;R))\Big)\rightarrowtail
\aut(A( \Sigma Y))\overset{}{
	\twoheadrightarrow}\aut\big(H_{q+1}( Y;R)\big)\times \aut\big(A( (\Sigma Y)^{q-1})\big)$$
$$\mathrm{Hom}\Big(H_{q+1}( Y;R),\T_{q}(H_{<q}(Y;R))\Big)\rightarrowtail
\aut_{*}(A( \Sigma Y))\overset{}{
	\twoheadrightarrow}\aut_{*}\big(A( (\Sigma Y)^{q-1})\big)$$
where  $(\Sigma Y)^{q-1}$  is  the $(q-1)$-skeleton  of the space $ \Sigma Y$.
\end{introtheorem}
For instance, using  the above short exact sequence   we can show the following results

$$\aut(A(\Bbb S^{n+1}\vee\Bbb S^{q+1}))\cong\Bbb Z_{2}\times \Bbb Z_{2}\,\,\,\,\,\,\,\,\,\,\,\,\,\,\,\,\,,\,\,\,\,\,\,\,\aut_*(A(\Bbb S^{n+1}\vee\Bbb S^{q+1}))\cong\Bbb Z_{2}\,\,\,\,\,\,\,\,\,\,\,\,,\,\,\,\,\,\text{ if } \,\,q\not\equiv 0\,\, (mod\,\, n)$$
$$\Bbb Z\rightarrowtail\aut(A(\Bbb S^{n+1}\vee\Bbb S^{q+1}))\overset{\mathrm{}}{ \twoheadrightarrow}\Bbb Z_{2}\times \Bbb Z_{2}\,\,,\,\,\,\,\Bbb Z\rightarrowtail\aut_{*}(A(\Bbb S^{n+1}\vee\Bbb S^{q+1}))\overset{\mathrm{}}{ \twoheadrightarrow} \Bbb Z_{2}\,\,,\,\,\text{ if } \,\,q\equiv 0\,\, (mod\,\, n)$$

\medskip

Moreover let  $R \subseteq \Bbb Q$  be  a subring with least non-invertible prime $p$, using the Anick  mode theory \cite{An1,An2}, if
$X$ be  $r$-connected CW-complex of dimension $n+1$ and    $n<q\leq k$, where $k < min(r + 2p-3,rp-1)$, then  we prove
\begin{introtheorem}
	Let $Y$  be the space in (\ref{020}). The homomorphisms
	$$ \mathcal{E}(Y_{R}) \to  \mathcal{E}(A(Y_{R})) \,\,\,\,\,\,\,\,\,\,\,\,,\,\,\,\,\,\,\, \mathcal{E}_{*}(Y_{R}) \to  \mathcal{E}_{*}(A(Y_{R}))$$
	are  injective, where $Y_{R}$ denotes the $R$-localised of $Y$.
\end{introtheorem}

 The paper is organised as follows.  In section 2, we recall briefly the notion of Adams-Hilton  model  associated to a given simply connected space $Y$ and as well as the  Bott-Sameson theorem concerning the  Adams-Hilton  model of the space $\Sigma Y$. In section 3 we establish theorem 3 and  in section 4,  we recall  the notion of the homotopy between chain algebra morphisms and we prove theorem 2 and some of its corollaries.
\section{Adams-Hilton  model and Bott-Sameson theorem}
Given   a simply connected  CW-complex $Y$.  The Adams-Hilton  model of $Y$ is  a chain algebra morphism  $$\Theta_{Y}: (\T( V),\partial)\to C_{*}(\Omega Y,R)$$
such that
$$H_{*}(\Theta_{Y}): H_{*}(\T( V),\partial)\to H_{*}(\Omega Y,R)$$
is an isomorphism of graded algebras and such as
\begin{equation}\label{0114}
H_{i-1}(V,d)\cong H_{i}(Y,R)\,\,\,\,\,\,\,\,\,\,\,\,,\,\,\,\,\,\,\,\text { as graded modules}
\end{equation}
Here $C_{*}(\Omega Y,R)$ denotes the complex of the non-degenerate cubic chains equipped with the multiplication induced by the composition of loops and $d:V\to V$ is the linear part of the differential $\partial$ defined by
$$\partial(v)-d(v)\in \T^{\geq 2}( V)$$
where $\T^{\geq 2}( V)$ is the graded $R$-module of the decomposable elements, i.e., the elements of $\T( V)$ of length $\geq 2$. We denote by $A(Y)$ the chain algebra $(\T( V),\partial)$.

\bigskip

Let  $\Sigma Y$ denote the suspension of $Y$. If the map $\sigma:Y\to\Omega\Sigma Y$ is the adjoint  of  $id_{\Sigma Y}$, then it induces a  homomorphism of graded module $\sigma_*:H_*(Y;R)\to H_{*}(\Omega\Sigma Y;R)$
which can be extend, by  virtue of the universal property of the free chain algebra,  to a homomorphism $$\T(\sigma_*):\T(H_*(Y;R))\to H_{*}(\Omega\Sigma Y;R)$$
of a graded algebra.

\noindent A theorem,  due to Bott and Sameson \cite{BS}, asserts that  under the assumption that the homology $H_{*}(Y;R)$ is a  free graded $R$-module,   $\T(\sigma_*)$ is an isomorphism of $R$-algebra.    Therefore the chain algebra $\big(\T(H_*(Y;R)),0\big)$ with the trivial differential can be considered  as an Adams-Hilton model for the space $ \Sigma Y$, i.e.,
\begin{equation}\label{a1}
A( \Sigma Y)=\big(\T(H_*(Y;R)),0\big)
\end{equation}
By virtues of the properties of the Adams-Hilton model we derive 
\begin{equation}\label{a2}
H_*(\Omega \Sigma Y;R)=H_*(A( \Sigma Y))=\T(H_*(Y;R))
\end{equation}
\begin{remark}
\label{r1}
It is important to mention here that as  the  graded $R$-module $H_{*}(Y;R)$ is assumed to be free, the two relations (\ref{a1}) and (\ref{a2}) imply that the Adams-Hilton model $(\T( V),0)$  of $\Sigma Y$  satisfies 
\begin{equation}\label{44}
V_{i}\cong H_{i-1}(Y;R)\,\,\,\,\,\,\,\,\,\,\,,\,\,\,\,\,\,\,\, \forall i\geq 2
	\end{equation}
\end{remark}

\section{ the group of the graded  algebra automorphisms of  tensor algebra $\T(V)$}
Let   $\T(V_q\oplus V_{\leq n})$, where $q>n$,  be  a tensor algebra (considered as 1-connected chain algebra  with trivial differential).  Let us denote by  $\aut(\T(V_q\oplus V_{\leq n}))$ the group of    chain (graded) algebra automorphisms of $\T(V_q\oplus V_{\leq n})$.

\noindent If $\alpha\in \aut(\T(V_q\oplus V_{\leq n}))$, then it induces the following homomorphism
\begin{equation}\label{10}
\alpha_k:V_k\to V_k\oplus\T_k( V_{\leq n})\,\,\,\,\,\,\,\,\,\,\,\,\,\,\,\,\,,\,\,\,\,\,\,\,\,\,\,\,\,\,\,\,\,\,\,k\leq q
\end{equation}
so define $\widetilde{\alpha}_k:V_k\to V_k$ such that $\alpha(v)-\widetilde{\alpha}_k(v)\in \T_k( V_{\leq n})$. Clearly $\widetilde{\alpha}_k$ is an automorphism of $V_k$. Hence denote by   $\aut_*(\T(V_q\oplus V_{\leq n}))$ the subgroup of $\aut(\T(V_q\oplus V_{\leq n}))$  of  the elements $\alpha$ such that $\widetilde{\alpha}_k=id$ for all $k\leq q$.

The aim of this section is to establish the following theorem
\begin{theorem}
	\label{t2}
If $\T(V)$ is an  1-connected free graded tensor algebra, then we have the following two short exact sequences of groups
\begin{eqnarray}
\label{-77}
\mathrm{Hom}(V_{q},  \T_{q}(V_{\leq n}))\rightarrowtail\aut(\T(V_{q}\oplus V_{\leq n}))\overset{\mathrm{}}{ \twoheadrightarrow}\aut(V_{q})\times \aut(\T (V_{\leq
	n}))
\end{eqnarray}
\begin{eqnarray}
\label{-770}
\mathrm{Hom}(V_{q},  \T_{q}(V_{\leq n}))\rightarrowtail\aut_*(\T(V_{q}\oplus V_{\leq n}))\overset{\mathrm{}}{ \twoheadrightarrow} \aut_*(\T (V_{\leq
	n}))
\end{eqnarray}
\end{theorem}
\begin{proof}
Let $(\T(V_q\oplus V_{\leq n})),$ where  $q> n$, be a free graded tensor. Define the  map $$\mathrm{g}:\aut(\T(V_q\oplus V_{\leq n}))\to\aut(V_{q})\times \aut(\T (V_{\leq
	n}))$$
by setting:
\begin{equation}\label{001}
\mathrm{g}(\alpha)=(\widetilde{\alpha}_{q},\alpha_{n})
\end{equation}
where $\widetilde{\alpha}_{q}:V_{q}\to V_{q}$ is   as above and where $\alpha_{n}$ is the restriction of  $\alpha$ to $\T(V_{\leq n})$.	

It is easy to see that $\mathrm{g}$ is a surjective  morphism of groups. Indeed,  let $(\xi,\gamma)\in \aut(V_{q})\times \aut(\T (V_{\leq
	n}))$.  Define $\alpha:\T(V_q\oplus V_{\leq n})
\rightarrow\T(V_q\oplus V_{\leq n})$ by setting:
\begin{eqnarray}
\alpha(v)=\xi(v)\,\,\,,\,\,
\text{and}\,\,\,\,\,
\alpha=\gamma\text{ on }V_{\leq
	n}.\label{0113}
\end{eqnarray}
Clearly we have $\widetilde{\alpha}_{q}=\xi$. Hence using (\ref{001}) we derive  $\mathrm{g}(\alpha)=(\xi,\gamma)$.

 Finally the following relations
\begin{equation}
\mathrm{g}(\alpha\circ\alpha')=(\widetilde{\alpha\circ\alpha'}_{q},\alpha_{ n}\circ \alpha'_{ n})=(\widetilde{\alpha}_{q},\alpha_{n})\circ (\widetilde{\alpha'}_{q},\alpha'_{ n})=\mathrm{g}(\alpha)\circ\mathrm{g}(\alpha')\nonumber
\end{equation}
assure that $\mathrm{g}$ is a homomorphism of groups.

\medskip

Consequently we obtain the following short exact sequence of groups
\begin{eqnarray}\label{-6}
\ker\,\mathrm{g}\rightarrowtail\aut(\T(V_{q}\oplus V_{\leq n}),0)\overset{\mathrm{g}}{ \twoheadrightarrow}\aut(V_{q})\times \aut(\T (V_{\leq
	n}))
\end{eqnarray}

Next let us determine  $\ker\,\mathrm{g}$. By  the formula (\ref{001})  we can write:
\begin{equation}\label{035}
\ker\,\mathrm{g}=\Big\{\alpha\in \aut(\T(V_q\oplus V_{\leq n}))\,\,\,\, \mid \,\,\,\,\, \widetilde{\alpha}_{q}=id_{V_q}\,\,\,\,\,,\,\,\,\,\,\alpha_{ n}=id_{\T_{}(V_{\leq
		n})}\Big\}
\end{equation}
therefore for every  $\alpha\in \ker\,\mathrm{g}$ we have:
\begin{eqnarray}\label{-1}
\alpha(v)&=&v+z_{v},\,\,\,\,\,\,\,\,\,\,\, z_{v}\in \T_{q}(V_{\leq n})\nonumber\\
\alpha_{ n}&=&id_{\T(V_{\leq
		n})}
\end{eqnarray}
So define the map $\Psi:\ker\,\mathrm{g}\to \mathrm{Hom}(V_{q},  \T_{q}(V_{\leq n}))$ by setting
\begin{equation}\label{-2}
\Psi(\alpha):V_{q}\to   \T_{q}(V_{\leq n})\,\,\,\,\,\,\,\,\,\,\,\,\,\,\,\,\,\,,\,\,\,\,\,\,\,\,\,\,\,\,\,\,\,\Psi(\alpha)(v)=z_{v}
\end{equation}
On one hand  the relations (\ref{-1}) and (\ref{-2}) imply that
\begin{equation*}\label{-3}
\alpha\circ\alpha'(v)=\alpha(v+z'_{v})=v+z_{v}+z'_{v}
\end{equation*}
hence $\Psi(\alpha\circ\alpha')(v)=z_{v}+z'_{v}$. On the other hand we have
\begin{equation*}\label{-4}
(\Psi(\alpha)+\Psi(\alpha'))(v)=\Psi(\alpha)(v)+\Psi(\alpha')(v)=z_{v}+z'_{v}
\end{equation*}
Therefore $\Psi(\alpha\circ\alpha')=\Psi(\alpha)+\Psi(\alpha')$ implying that $\Psi$ is a homomorphism of group.

Now let $ \alpha\in \ker\,\mathrm{\Psi}$, then  $\Psi(\alpha)=0$ implying that $\Psi(\alpha)(v)=z_{v}=0$ and according to (\ref{-1}), it follows that $\alpha=id$. Hence $\Psi$ is injective. Finally let $f\in\mathrm{Hom}(V_{q},  \T_{q}(V_{\leq n}))$, define
\begin{equation}\label{-5}
\alpha(v)=v+ f(v)\,\,\,\,\,\,\,\,\,\,\,\,\,\,\,,\,\,\,\,\,\,\,\,\,\,\,\,\,\,\,\
\alpha_{ n}=id_{\T(V_{\leq
		n})}
\end{equation}
By definition   (\ref{-2}) we have $\Psi(\alpha)(v)=f(v)$, so $\Psi(\widetilde{\gamma}))=f$. It follows that $\Psi$ is surjective,  consequently $\Psi$ is an  isomorphism of groups.

\medskip

Summarising  the short exact sequence (\ref{-6})  becomes
$$
\mathrm{Hom}(V_{q},  \T_{q}(V_{\leq n}))\rightarrowtail\aut(\T(V_{q}\oplus V_{\leq n}))\overset{\mathrm{g}}{ \twoheadrightarrow}\aut(V_{q})\times \aut(\T (V_{\leq
	n}))$$

\medskip

Next let $\mathrm{\hat{g}}$ denote the restriction of the homomorphism $\mathrm{g}$ to the subgroup $\aut_*(\T(V_{q}\oplus V_{\leq n}))$. As $\aut_*(\T(V_q\oplus V_{\leq n}))$ is formed by the  elements $\alpha$ such that $\widetilde{\alpha}_k=id$ for all $k\leq q$, it follows that
\begin{equation}\label{40}
\mathrm{\hat{g}}(\alpha)=(id_{V_{q}},\alpha_{n})\,\,\,\,\,\,\,\,\,\,\,\,\,,\,\,\,\,\,\,\,\,\,\,\,\,\,\,\,\alpha_{n}\in \aut_*(\T(V_{\leq n}))
\end{equation}
Hence we define $\mathrm{\widetilde{g}}:\aut_*(\T(V_{q}\oplus V_{\leq n}))\overset{}{ \twoheadrightarrow} \aut_*(\T( V_{\leq n})$ by $\mathrm{\widetilde{g}}(\alpha)=\alpha_{n}$. The map  $\widetilde{g}$ is a surjective homomorphism.  Indeed,  if $\gamma\in \aut_*(\T(V_{\leq n}))$, then  $(id_{V_{q}},\gamma)\in \aut(V_{q})\times \aut(\T (V_{\leq
	n}))$. As the homomorphism  $\mathrm{g}$ is surjective, there exist $\alpha\in \aut(\T(V_{\leq n}))$ such that
$\mathrm{g}(\alpha)=(id_{V_{q}},\gamma) $. Using  (\ref{001}) we deduce that
$$(id_{V_{q}},\gamma) =(\widetilde{\alpha}_{q},\alpha_{n})$$
implying that  $\alpha\in \aut_*(\T(V_{\leq n}))$ and $\mathrm{\widetilde{g}}(\alpha)=\gamma$.

\noindent Now from  (\ref{035}) we have
$ \ker\,\mathrm{g}=\ker\,\mathrm{\widetilde{g}}$ and since $ \ker\,\mathrm{g}\cong \mathrm{Hom}(V_{q},  \T_{q}(V_{\leq n}))$, we obtain the following short exact sequence
\begin{equation*}
\label{41}
\mathrm{Hom}(V_{q},  \T_{q}(V_{\leq n}))\rightarrowtail\aut_*(\T(V_{q}\oplus V_{\leq n}))\overset{\mathrm{}}{ \twoheadrightarrow} \aut_*(\T (V_{\leq
	n}))
\end{equation*}
\end{proof}
\begin{corollary}\label{c03}
If  $\Sigma Y$ is  a simply connected space of dimension $q+1$ and let  $X=(\Sigma Y)^{q}$ denotes   the $q$ skeleton  of  $ \Sigma Y$. Then  the following short sequence of groups is exact.
$$\mathrm{Hom}\Big(H_{q+1}( Y;R),\T_{q}(H_{<q}(Y;R))\Big)\rightarrowtail
\aut(A( \Sigma Y))\overset{}{
	\twoheadrightarrow}\aut\big(H_{q+1}( Y;R)\big)\times \aut\big(A( (\Sigma Y)^{q-1})\big)$$
\begin{eqnarray}
\label{43}
\mathrm{Hom}\Big(H_{q+1}( Y;R),\T_{q}(H_{<q}(Y;R))\Big)\rightarrowtail
	\aut_{*}(A( \Sigma Y))\overset{}{
		\twoheadrightarrow}\aut_{*}\big(A( (\Sigma Y)^{q-1})\big)
		\end{eqnarray}
	\end{corollary}
\begin{proof}
First  the Adams-Hilton of the space $\Sigma Y$ is on the form $\T(V_{q}\oplus V_{\leq q-1})$ with trivial differential. Next we derive the tow sequences in (\ref{43})  by applying theorem  \ref{t2} and using  the relations (\ref{a1}), (\ref{a2})  and  (\ref{44}).
\end{proof}

\begin{corollary}\label{c01}
	Let $V_q=\{v_{q}\}$ be  the free $R-$module  of rank 1 and let  $\T(V_q\oplus V_{\leq n})$, where $q>n$,  be  a free  tensor algebra. Then  the following short sequence of groups is exact.
	\begin{eqnarray}
	\label{002}
	\T_{q}(V_{\leq n})\rightarrowtail\aut(\T(V_{q}\oplus V_{\leq n}))\overset{\mathrm{}}{ \twoheadrightarrow}\aut(R)\times \aut(\T (V_{\leq
		n}))
	\end{eqnarray}
	$$	\T_{q}(V_{\leq n})\rightarrowtail\aut_*(\T(V_{q}\oplus V_{\leq n}))\overset{\mathrm{}}{ \twoheadrightarrow} \aut_*(\T (V_{\leq
		n}))$$
	\end{corollary}
\begin{proof}
The sequence (\ref{002}) can be deduced form the exact sequence  (\ref{-77}) by observing that  $V_q=\{v_{q}\}\cong R$ then $\mathrm{Hom}(V_{q},  \T_{q}(V_{\leq n}))\cong \T_{q}(V_{\leq n})$ and $\aut(V_{q})\cong\aut(R)$
\end{proof}
As an illustration of corollary \ref{c01} we give the following example
\begin{example}
	
Let $R=\Bbb Z$ and $X=\Bbb S^{n+1}$ be the sphere of dimension $n+1$.  Let
	$Y=\Bbb S^{n+1}\vee\Bbb S^{q+1}$, where $q>n$ and where the $q+1$-cell is trivially attached to  $\Bbb S^{n+1}$.
	
\noindent Recall that  $A( \Sigma Y)$ (respect. $A( \Sigma Y)$) denotes  the Adams-Hilton model of the suspension of the space $Y$ (respect. of $X$ (see \ref{a1}))  and  let  $\aut(A( \Sigma Y))$ (respect. $\aut(A( \Sigma Y))$)  denote   the group  of the graded automorphisms of the  free tensor algebra $A( \Sigma Y)$ (respect. $A( \Sigma X)$).

The Adams-Hilton of $S^{n+1}=\Sigma S^{n}$  and of $Y=\Bbb S^{n+1}\vee\Bbb S^{q+1}$ are respectively   $$A( S^{n+1})=\T(H_*(S^{n};\Bbb Z))$$
$$A( \Bbb S^{n+1}\vee\Bbb S^{q+1})=A( \Sigma(\Bbb S^{n}\vee\Bbb S^{q}))\cong \T(H_*(S^{n}\vee\Bbb S^{q});\Bbb Z))$$

Recall that
$$
H_q(S^{n};\Bbb Z)=
\begin{cases}
0,\,\,\text{ if } \,\,q\neq n \\
\Bbb Z,\,\, \text{ if } \,\,q= n \
\end{cases}
$$
and
$$
H_*(S^{n}\vee\Bbb S^{q};\Bbb Z)=H_*(S^{n};\Bbb Z)\oplus H_*(\Bbb S^{q};\Bbb Z)
$$
Define the graded abelian group $V_{q}\oplus V_{\leq n}$ by
$$V_{q}\cong H_q(\Bbb S^{q};\Bbb Z)\cong \Bbb Z\,\,\,\,\,,\,\,\,\,\,\,V_{n}\cong H_n(\Bbb S^{n};\Bbb Z)\cong \Bbb Z\,\,\,\,\,,\,\,\,\,\,\,V_{i}=0\,\,\,,\,\,\,i\leq n-1$$
Therefore  we obtain
$$A( S^{n+1})=\T(V_{\leq n})\,\,\,\,\,\,\,\,,\,\,\,\,\,\,\,\,\,\,A( \Bbb S^{n+1}\vee\Bbb S^{q+1})=A( \Sigma(\Bbb S^{n}\vee\Bbb S^{q}))\cong \T(V_{q}\oplus V_{\leq n} ;\Bbb Z)$$
Applying  corollary \ref{c01} we get
\begin{eqnarray}
\label{003}
\T_{q}(H_*(S^{n};\Bbb Z))\rightarrowtail\aut(A(\Bbb S^{n+1}\vee\Bbb S^{n+2}))\overset{\mathrm{}}{ \twoheadrightarrow}\aut(\Bbb Z)\times \aut(\T(H_*(S^{n};\Bbb Z)))
\end{eqnarray}
Let us compute $ \aut(\T(H_*(S^{n};\Bbb Z)))$. Indeed, we have
$$\aut(\T(H_*(S^{n};\Bbb Z)))=\aut(\T(V_{\leq n}))=\aut(\T(V_{n}\oplus V_{\leq n-1}))$$
Applying  again corollary \ref{c01} it follows
\begin{eqnarray}
\label{004}
\T_{q}(V_{\leq n-1})\rightarrowtail\aut(\T(V_{n}\oplus V_{\leq n-1}))\overset{\mathrm{}}{ \twoheadrightarrow}\aut(V_{n})\times \aut(\T(H_*(V_{\leq n-1}))
\end{eqnarray}
and taking in account that $V_{\leq n-1}=0$ we obtain
$$\aut(\T(V_{n}\oplus V_{\leq n-1}))\cong\aut(V_{n})$$
Hence
$$\aut(\T(H_*(S^{n};\Bbb Z)))\cong\aut(V_{n}) \cong\aut(\Bbb Z)=\Bbb Z_{2}$$
Consequently the exact sequence (\ref{003}) becomes
\begin{eqnarray}
\label{006}
\T_{q}(H_*(S^{n};\Bbb Z))\rightarrowtail\aut(A(\Bbb S^{n+1}\vee\Bbb S^{q+1}))\overset{\mathrm{}}{ \twoheadrightarrow}\Bbb Z_{2}\times \Bbb Z_{2}
\end{eqnarray}
But we have
$$
\T_{q}(H_*(S^{n};\Bbb Z))=
\begin{cases}
0,\,\,\text{ if } \,\,q\not\equiv 0\,\, (mod\,\, n)\\
\Bbb Z,\,\, \text{ if } \,\,q\equiv 0 \,\,(mod\,\, n) \
\end{cases}
$$
implying
\begin{equation*}
\label{007}
\aut(A(\Bbb S^{n+1}\vee\Bbb S^{q+1}))\cong\Bbb Z_{2}\times \Bbb Z_{2}\,\,\,\,\,\,\,\,\,\,\,\,\,\,\,\,\,\,\,,\,\,\,\,\,\,\,\text{ if } \,\,q\not\equiv 0\,\, (mod\,\, n)
\end{equation*}
\begin{equation*}
\label{009}
\Bbb Z\rightarrowtail\aut(A(\Bbb S^{n+1}\vee\Bbb S^{q+1}))\overset{\mathrm{}}{ \twoheadrightarrow}\Bbb Z_{2}\times \Bbb Z_{2}\,\,\,\,\,,\,\,\,\,\,\,\,\text{ if } \,\,q\equiv 0\,\, (mod\,\, n)
\end{equation*}
Finally a similar computation shows that
\begin{equation*}
\aut_*(A(\Bbb S^{n+1}\vee\Bbb S^{q+1}))\cong\Bbb Z_{2}\,\,\,\,\,\,\,\,\,\,\,\,\,\,\,\,\,\,\,,\,\,\,\,\,\,\,\text{ if } \,\,q\not\equiv 0\,\, (mod\,\, n)
\end{equation*}
\begin{equation*}
\Bbb Z\rightarrowtail\aut_{*}(A(\Bbb S^{n+1}\vee\Bbb S^{q+1}))\overset{\mathrm{}}{ \twoheadrightarrow} \Bbb Z_{2}\,\,\,\,\,,\,\,\,\,\,\,\,\text{ if } \,\,q\equiv 0\,\, (mod\,\, n)
\end{equation*}
Notice that $\aut_{*}(A(\Bbb S^{n+1}\vee\Bbb S^{q+1}))$ is a normal subgroup of $\aut(A(\Bbb S^{n+1}\vee\Bbb S^{q+1}))$  and in the two cases the quotient group is
$$\frac{\aut(A(\Bbb S^{n+1}\vee\Bbb S^{q+1})) }{ \aut_{*}(A(\Bbb S^{n+1}\vee\Bbb S^{q+1}))}\cong \Bbb Z_{2}$$
\end{example}

\bigskip

In the second part of this paper we shall generalize the above results to the case when the differential given in $(\T(V),\partial)$ is not necessary trivial. For this purpose we need the notion of the homotopy between  chain algebra morphisms  which is analogous in many respects  to the topological notion of homotopy.
\section{The Group of Homotopy self-equivalences of chain algebra morphisms}
\subsection{Homotopy of chain algebra morphisms }(See \cite{Ba} page 48 for more details)
Let $(\T(V),\partial)$ be a 1-connected free  chain algebra. Define the free algebra    $\T(V'\oplus V''\oplus sV)$, where $V', V''$ are two copies isomorphic to $V$ and where $sV$ is the (de)suspension of $V$. Then we define:
\begin{eqnarray}
\label{01}
i',i'':\T(V)\to\T(V'\oplus V''\oplus sV)\,\,\,\,\,\,\,,\,\,\,\,\,\,\,i'(v)=v'\,\,\,\,\,\,\,,\,\,\,\,\,\,\,i''(v)=v''
\end{eqnarray}
where $v'\in V', v''\in V"$ are the two elements  corresponding to $v\in V$. Now define a graded homomorphism, of degree 1, $S:\T(V)\to\T(V'\oplus V''\oplus sV) $ as the unique graded homomorphism which satisfies the following two conditions
\begin{eqnarray}
\label{02}
S(v)=sv\,\,\,\,\,\,\,\,\,,\,\,\,\,\,\,\,\,\,\,S(x.y)=S(x).(i''(y)+(-1)^{\vert x\vert}i'(x)S(y)\,\,\,\,\,,\,\,\,\,\forall x,y\in \T(V)
\end{eqnarray}
Next we define the differential $D$ on $\T(V'\oplus V''\oplus sV)$ by setting
\begin{equation}\label{8}
D(sv)=v''-v'-S(\partial v)\,\,\,\,\,\,\,\,\,,\,\,\,\,\,\,\,\,\,\, D(v')=i'(\partial v)\,\,\,\,\,\,\,\,\,,\,\,\,\,\,\,\,\,\,\, D(v'')=i''(\partial v).
\end{equation}
$(\T(V'\oplus V''\oplus sV),D)$ is called the the   cylinder chain algebra of $(\T(V),\partial)$.
\begin{definition}
	\label{d1}
	A  homotopy   between two chain algebra morphisms $\alpha_{1}, \alpha_{2}: (\T(V),\partial)\to (\T(V),\partial)$  is a chain algebra morphism
	$$F \colon ( \T(V'\oplus V''\oplus sV),D)\to (\T(V),\partial)$$
	such us $F\circ i'(v)=F(v')=\alpha_{1}(v)$ and $F\circ i''(v)=F(v'')=\alpha_{2}(v)$.
\end{definition}
\begin{definition}
	\label{d2}
	A  chain algebra morphism $\alpha_{1}: (\T(V),\partial)\to (\T(V),\partial)$ is called a self-homotopy equivalence, if there exists a chain algebra morphism   $\alpha_{2}: (\T(V),\partial)\to (\T(V),\partial)$  such that $\alpha_{1}\circ \alpha_{2}$ and $\alpha_{2}\circ \alpha_{1}$ are homotopic to the identity.
\end{definition}
\begin{definition}
	\label{d6}
	Let  $\E(\T(V))$ denote  the group, equipped with the composition of chain algebra morphisms,  of the  self-homotopy equivalences of $(\T(V),\partial)$  and  let  $\E_{*}(\T(V))$ denote its  subgroup   of   the  elements inducing the identity on the graded module of the indecomposable $V_{*}$
\end{definition}

\medskip
Thereafter we will need the following lemma
\begin{lemma}
	\label{l3} Let $q>n$ and let   $V = V_{q}\oplus V_{\leq n}$ and $\alpha_{1},\alpha_{2} \colon   (\T(V),\partial) \to (\T(V),\partial)$  be two chain algebra morphisms satisfying:
	$$\alpha_{1}(v)=v+z_{1}, \,\,\,\,\,\,\,\,\alpha_{2}(v)=v+z_{2} \,\text{ on $V_{q}$}  \,\,\,\,
	\text{ and }\,\,\,\, \alpha=\alpha'=\mathrm{id}
	\,\,\text{ on $V_{\leq n}$}.$$
	Assume that $z_{1}-z_{2}=\partial(u)$,
	where $u\in \T_{q+1}(V)$. Then $\alpha_{1}$ and $ \alpha_{2}$
	are  homotopic.
\end{lemma}
\begin{proof}
	Define $F$ by setting
	\begin{eqnarray}
	\label{9}
	F(v') &=& v+z_{1}, \,\,\,\,\,F(v'')=v+z_{2} \hbox{\ \ and \ \ }F(sv)=u \hbox{\ \ for \ \ } v \in V^{q}\nonumber \\
	F(v') &=& v, \,\,\,\,\,\,\,\,\,\,\,\,\,\,\,\,\,\,F(v'')=v \,\,\,\,\,\,\,\,\,\,\,\,\hbox{\ \ and \ \ }F(sv)=0 \hbox{\ \ for \ \ } v \in V^{\leq n}.
	\end{eqnarray}
	then   $F$ is the needed homotopy.
\end{proof}
Let's start with the following remarks.
\begin{remark}\label{r0}
	If $(\T(V),0)$ is a 1-connected free chain algebra with trivial differential, then the notion of homotopy is simply  the  equality. Indeed, let $\alpha_{1}, \alpha_{2}: (\T(V),0)\to (\T(V),0)$ be two chain algebra morphisms and  assume that they are homotopic. By definition \ref{d1} there exist a chain   algebra morphism
	$$F \colon ( \T(V'\oplus V''\oplus sV),D)\to (\T(V),0)\,\,\,\,\,\,\,\,\,\,\,\,\,\,\,\,$$
	such us
	\begin{equation}\label{07}
	F\circ i'(v)=F(v')=\alpha_{1}(v)\,\,\,\,\,\,\,\,\,\,\,\,\,\,\,\,,\,\,\,\,\,\,\,\,\,\,\,\,\,\,\,\,F\circ i''(v)=F(v'')=\alpha_{2}(v)
	\end{equation}
	As the differential $\partial$ is trivial and $F$ is a chain algebra, it follows that
	\begin{equation}\label{04}
	F\circ D=0
	\end{equation}
	moreover the relations (\ref{8}) become
	\begin{equation}\label{008}
	D(sv)=v''-v'\,\,\,\,\,\,\,\,\,,\,\,\,\,\,\,\,\,\,\, D(v')=0\,\,\,\,\,\,\,\,\,,\,\,\,\,\,\,\,\,\,\, D(v'')=0.
	\end{equation}
	Therefore
	\begin{equation}\label{09}
	0=F\circ D(sv)=F(v'')-F(v').
	\end{equation}
	Finally according to (\ref{07}) we deduce that $\alpha_{1}(v)=\alpha_{2}(v)$
\end{remark}
\begin{remark}\label{r00}
	Let $(\T(V),0)$ be  a 1-connected free chain algebra with trivial differential.  By virtue of remark \ref{r0} we derive  that  the  group $\E(\T(V))$ is identified with the group $\aut(\T(V))$ and  $\E_{*}(\T(V))$ is identified with the subgroup $\aut_*(\T(V))$ introduce in the previous section.
\end{remark}
\subsection{The graded homomorphism $b_{*}$ and the groups $\mathcal{D}^{q}_{n}$}
\begin{definition}
	\label{d4}
	Let $(\T(V_q\oplus V_{\leq n}),\partial)$ be a 1-connected chain algebra where $q>n$. We define the homomorphism  $b_{q}: V_{q}\rightarrow
	H_{q-1}(\T (V_{\leq n})$ by setting:
	\begin{equation}
	b_{q}(v)=[\partial(v)]\label{37}
	\end{equation}
	Here $[\partial(v)]$  denotes
	the homology class of $\partial(v)\in
	\T_{q-1} (V_{\leq n})$.
\end{definition}
For every    1-connected chain algebra $(\T(V_q\oplus V_{\leq n}),\partial)$, the homomorphism $b_{q}$ is natural. Namely if $[\alpha]\in\E(\T(V_q\oplus V_{\leq n}))
$, then  the following diagram commutes:

\begin{equation}\label{11}
\begin{picture}(300,90)(00,30)
\put(60,100){$V_{q}\hspace{1mm}\vector(1,0){155}\hspace{1mm}V_{q}$}
\put(69,76){\scriptsize $b_{q}$} \put(238,76){\scriptsize $b_{q}$}
\put(66,96){$\vector(0,-1){40}$} \put(235,96){$\vector(0,-1){40}$}
\put(150,103){\scriptsize $\widetilde{\alpha}_{q}$} \put(135,50){\scriptsize
	$H_{q-1}(\alpha_{n})$} \put(50,46){$H_{q-1}(\T (V_{\leq n}))
	\hspace{1mm}\vector(1,0){100}\hspace{1mm}H_{q-1}(\T(V_{\leq n}))
	\hspace{1mm}$}
\end{picture}
\end{equation}
where
\begin{equation}\label{0}
\widetilde{\alpha}:(V_q\oplus V_{\leq n},d)\to (V_q\oplus V_{\leq n},d)
\end{equation}
is the graded
homomorphism induced by $\alpha$ on the chain complex of the  indecomposables and where
$\alpha_{n}:(\T( V_{\leq n}) ,\partial)\to (\T(V_{\leq
	n} ),\partial)$ is the restriction of $\alpha$. Here $d$ denotes the linear part of the differential $\partial$ defined by the relation
$$\partial-d:V_{n+1}\to \T_{n}^{\geq 2}( V_{})$$
\begin{definition}\label{d3}
	Given a 1-connected chain algebra $(\T(V_q\oplus V_{\leq n}),\partial)$ where  $q> n$ and set $V=V_q\oplus V_{\leq n}$. Let $\mathcal{D}^{q}_{n}$ be the subset  of $\aut(V_{q})\times \mathcal{E}(\T (V_{\leq
		n}))$ consisting of the couples $(\xi,[\alpha])$  making  the following
	diagram commutes
	\begin{equation}\label{120}
	\begin{picture}(300,90)(00,30)
	\put(60,100){$V_{q}\hspace{1mm}\vector(1,0){155}\hspace{1mm}V_{q}$}
	\put(69,76){\scriptsize $b_{q}$} \put(238,76){\scriptsize $b_{q}$}
	\put(66,96){$\vector(0,-1){38}$} \put(235,96){$\vector(0,-1){38}$}
	\put(145,103){\scriptsize $\xi$} \put(138,50){\scriptsize
		$H_{q-1}(\alpha)$} \put(50,46){$H_{q-1}(\T (V_{\leq n}))
		\hspace{1mm}\vector(1,0){100}\hspace{1mm}H_{q-1}(\T(V_{\leq n}))
		\hspace{1mm}$}
	\end{picture}
	\end{equation}
	Clearly $\mathcal{D}^{q}_{n}$ is a subgroup  of $\aut(V_{q})\times \mathcal{E}(\T (V_{\leq
		n}))$.
\end{definition}

\begin{remark}\label{r001}
	If $(\T(V),0)$ is a  1-connected free chain algebra with trivial differential,   then according to the relation (\ref{37}),  the homomorphism $b_q$ given in  the diagram (\ref{120}) is nil. Moreover we have
	$$H_{q-1}(\T(V_{\leq n}))=\T_{q-1}(V_{\leq n})\,\,\,\,\,\,\,\,\,\,\,\,,\,\,\,\,\,\,\,\,\,\,\,\,\,\,\,H_{q-1}(\alpha)=\alpha_{q-1}$$
	As a result  the group $\mathcal{D}^{q}_{n}$  which is constituting with  the pairs  $(\xi,\alpha)\in\aut(V_{q})\times \aut(\T (V_{\leq
		n}))$  making  the following
	diagram commutes (see definition  \ref{d3}).
	\begin{equation*}\label{12}
	\begin{picture}(300,90)(-10,30)
	\put(60,100){$ V_{q}\hspace{1mm}\vector(1,0){165}\hspace{1mm}V_{q}$}
	\put(69,76){\scriptsize $0$} \put(248,76){\scriptsize $0$}
	\put(66,96){$\vector(0,-1){38}$} \put(245,96){$\vector(0,-1){38}$}
	\put(155,103){\scriptsize $\xi$} \put(155,50){\scriptsize
		$\gamma_{q-1}$}
	\put(45,46){$\T_{q-1}(V_{\leq n})
		\hspace{1mm}\vector(1,0){145}\hspace{1mm}\T_{q-1}(V_{\leq n})
		\hspace{1mm}$}
	\end{picture}
	\end{equation*}
	Therefore  $\mathcal{D}^{q}_{n}$ is just the group  $\aut(V_{q})\times \aut(\T (V_{\leq
		n}))$ used in the previous section.
\end{remark}

\begin{proposition}\label{pp3}
The map $\mathrm{g}:\mathcal{E}(\T(V_q\oplus V_{\leq n}))\to\mathcal{D}^{q}_{n}$ given by:
$$\mathrm{g}([\alpha])=(\widetilde{\alpha}_{q},[\alpha_{n}])$$
is a surjective  homomorphism of groups
\end{proposition}
\begin{proof}
First it is well-known  that if two
chain morphisms  are homotopic, then they induce the same graded linear
maps on the chain complex of the   indecomposables, i.e.,
$\widetilde{\alpha}=\widetilde{\alpha'}$, moreover
$\alpha_{n},\alpha'_{n}$ are homotopic and by using the diagram (\ref{11}) we deduce that the map $\mathrm{g}$
is
well-defined.

\noindent Next let $(\xi,[\alpha_{n}])\in \mathcal{D}^{q}_{n}$. Recall
that, in the diagram (\ref{12}), we have:
\begin{eqnarray}
\label{117}H_{q-1}(\alpha_{ n})\circ
b_{q}(v)&=&\alpha_{ n}\circ\partial(v)+\mathrm{Im}\,\partial_{\leq n}\nonumber\\
b_{q}\circ
\xi_{q}(v)&=&\partial\circ\xi(v)+\mathrm{Im}\,\partial_{\leq
n}\label{103}
\end{eqnarray}
where $\partial_{\leq n}:\T_{q} (V_{\leq n})\to \T_{q-1} (V_{\leq n})$.

\noindent Since by definition \ref{d3} this diagram commutes, the element
$(\alpha_{ n}\circ\partial-\partial\circ\xi)(v)\in\mathrm{Im}
\,\partial_{\leq  n}$. As a consequence there exists $u_{v}\in
\T_{q} (V^{\leq n})$ such that:
\begin{eqnarray}
(\alpha_{n}\circ\partial-\partial\circ\xi)(v)=\partial_{\leq
n}(u_{v})\label{123}.
\end{eqnarray}
Thus we define $\alpha:(\T(V_q\oplus V_{\leq n})
,\partial)\rightarrow(\T(V_q\oplus V_{\leq n}),\partial)$ by setting:
\begin{eqnarray}
\alpha(v)=\xi(v)+u_{v}\,\,\,,\,\,
\text{and}\,\,\,\,\,
\alpha=\alpha_{n}\text{ on }V_{\leq
n}.\label{113}
\end{eqnarray}
As $\partial(v)\in \T_{q-1} (V_{\leq n})$ then, by
(\ref{123}), we get:
\begin{eqnarray}
\partial\circ\alpha(v)=\partial(\xi(v))+\partial_{ n}(u_{v})=\alpha_{ n}\circ\partial(v)=
\alpha\circ\partial(v)
\end{eqnarray}
So  $\alpha$ is a  chain algebra morphism.   Now as
$u_{v}\in \T_{q} (V_{\leq n})$ and $q>n$,  the homomorphism
$\widetilde{\alpha}_{q}:V_{q}\to V_{q}$ coincides with
$\xi$.

\noindent Then  it  is well-known (see \cite{An1}, \cite{Benk6} and \cite{Benk11}) that  any
 chain algebra morphism between two 1-connected chain algebras  inducing a graded
 isomorphism on the homology of  the chain complex of the  indecomposables (see \ref{0}) is a homotopy
equivalence. Consequently  $[\alpha]\in \mathcal{E}(\T (V))$. Therefore  $\mathrm{g}$ is onto.

\noindent Finally the following relations
\begin{eqnarray}
\mathrm{g}([\alpha][\alpha'])&=&\mathrm{g}([\alpha\circ\alpha'])=(\widetilde{\alpha\circ\alpha'}_{q},[\alpha_{ n}\circ \alpha'_{ n}])\nonumber\\
&=&(\widetilde{\alpha}_{q},[\alpha_{n}])\circ (\widetilde{\alpha'}_{q},[\alpha'_{ n}])=\mathrm{g}([\alpha])\circ\mathrm{g}([\alpha'])\nonumber
\end{eqnarray}
assure that $\mathrm{g}$ is a homomorphism of groups
\end{proof}
\subsection{Characterization of $\ker\,\mathrm{g}$}
Next by definition we have:
\begin{equation}\label{35}
\ker\,\mathrm{g}=\Big\{[\alpha]\in \mathcal{E}(\T(V_q\oplus V_{\leq n}))\,\,\,\, \mid \,\,\,\,\, \widetilde{\alpha}_{q}=id_{V_q}\,\,\,\,\,,\,\,\,\,\,[\alpha_{ n}]=[id_{\T(V_{\leq
   n})}]\Big\}
\end{equation}
therefore for every  $[\alpha]\in \ker\,\mathrm{g}$ we have:
\begin{eqnarray}\label{31}
\alpha(v)&=&v+z,\,\,\,\,\,\,\,\,\,\,\, z\in \T_{q}(V_{\leq n})\nonumber\\
\alpha_{ n}&\simeq&id_{\T(V_{\leq
   n})}
\end{eqnarray}
So define:
\begin{equation}\label{13}
\theta_{\alpha}:V_{q}\to \T_{q}(V_{\leq n})\,\,\,\ \text{ by }\,\,\,\theta_{\alpha}(v)=\alpha(v)-v
\end{equation}
Notice that the relations (\ref{31}) and (\ref{13}) imply that
\begin{equation}\label{15}
   \theta_{\alpha'\circ\alpha}=\theta_{\alpha'}+\theta_{\alpha}
\end{equation}
\begin{remark}\label{r01}
If the differential in the chain algebra $(\T(V_q\oplus V_{\leq n})$ is trivial, then according to remark \ref{r0}  the formula (\ref{31}) becomes
\begin{eqnarray}\label{031}
\alpha(v)&=&v+z,\,\,\,\,\,\,\,\,\,\,\, z\in \T_{q}(V_{\leq n})\nonumber\\
\alpha_{ n}&=&id_{\T(V_{\leq
		n})}
\end{eqnarray}
implying that the element $\theta_{\alpha}(v)=z$ is a cycle in $\T_{q}(V_{\leq n})$. Notice that   if the differenial is not nil, then  $\theta_{\alpha}(v)$ might be not a cycle. However  we have the following crucial lemma
\end{remark}
\begin{lemma}
\label{l1}
Let $[\alpha]\in \ker\,\mathrm{g}$. Then there exists $[\beta]\in \ker\,\mathrm{g}$ satisfying:
\begin{enumerate}
  \item $\theta_{\beta}(v)$ is a cycle in $\T_{q}(V_{\leq n})$  for every $v\in V_{q}$
  \item $\beta_{n}=id_{\T(V_{\leq
  		n})}$
  \item $[\beta]=[\alpha]$
\end{enumerate}
\end{lemma}
\begin{proof}
Write $V=V_q\oplus V_{\leq n}$. Since,   $[\alpha_{n}]=[id_{\T(V_{\leq
   n})}]$ there is a homotopy:
   $$F \colon ( \T(V'_{\leq n}\oplus V_{\leq n}''\oplus (sV)_{\leq n}),D)\to (\T(V_{\leq n}),\partial)$$
   such us for every $x\in \T(V)$ we have
    \begin{equation}\label{03}
   F\circ i'(x)=\alpha(x)\,\,\,\,\,\,\,\,\,\,\,\,\,\,\,,\,\,\,\,\,\,\,\,\,\,\,\,\,F\circ i''(x)=x.
   \end{equation}
Thus we define  $\beta$ by setting:
 \begin{equation}\label{05}
\beta(v)= \left\{
   \begin{array}{ll}
     \alpha(v)-F\big(S(\partial v)\big), & \,\,\,\,\,\,\,\,\hbox{for $v\in V_q$;} \\
     v, & \,\,\,\,\,\,\,\,\hbox{for $v\in V_{\leq n}$.}
   \end{array}
 \right.
 \end{equation}
 Notice that as $v\in V_q$, we deduce that $\partial v\in \T_{q-1}(V)$. It follows that $S(\partial v)\in \T_{q}(V'_{\leq n}\oplus V_{\leq n}''\oplus (sV)_{\leq n})$, so the element $F\big(S(\partial v)\big)\in \T_{q}(V_{\leq n})$.

\noindent Let us prove that $\beta$ is a chain algebra morphism. Indeed first, for $v\in V_{q}$ , using the relations  (\ref{8}), we deduce that
\begin{equation}\label{013}
0=D^2(sv)=D(v''-v')-DS(\partial v)
\end{equation}
Now by virtues of  (\ref{03}) and (\ref{05}) we get
\begin{eqnarray}\label{06}
\partial(\beta(v))&=&\partial\circ\alpha (v)-\partial\circ F\big(S(\partial v)\big)\\
&=&\partial\circ\alpha  (v)-F(DS(\partial v))\nonumber\\
&=&\partial\circ\alpha (v)-F(D(v'-v''))\nonumber\\
&=&\partial\circ\alpha  (v)-F\circ i'(\partial v))+F\circ i''(\partial v))\nonumber\\
&=&\partial\circ\alpha (v)-\alpha\circ\partial(v)+\partial(v)\nonumber\\
&=&\partial(v)=\beta(\partial(v))
\end{eqnarray}
Here we use (\ref{013}) and the fact that $\partial(v)\in \T(V _{\leq n})$ and $\beta$ is the identity on $V_{\leq n}$.

\medskip
Consequently
\begin{equation*}
\partial(\theta_{\beta}(v))=\partial(\beta(v)-v)=\partial(\beta(v))-\partial(v)=\partial(v)-\partial(v)=0
\end{equation*}
Thus  $\theta_{\beta}(v)$ is a cycle in $\T_{q}(V_{\leq n})$ and  $\beta_{n}=id_{\T(V_{\leq
		n})}$. Next let us define
   $$G\colon ( \T(V'\oplus V''\oplus (sV)),D)\to (\T(V),\partial)$$
   by setting
\begin{eqnarray}
\label{08}
G(v') &=& \alpha(v), \,\,\,\,\,\hbox{\ \ on \ \ } V_{q}' \\
G(v') &=& \beta(v), \,\,\,\,\, \hbox{\ \ on \ \ } V_{q}'' \nonumber\\
G(sv) &=& 0, \,\,\,\,\, \,\,\,\,\,\,\,\,\hbox{\ \ on \ \ } (sV)_{q}\nonumber \\
G &=& F, \,\,\,\,\,\,\,\,\,\,\,\, \hbox{\ \ on \ \ }   V'_{\leq n}\oplus V_{\leq n}''\oplus (sV)_{\leq n}\nonumber
\end{eqnarray}
Using  (\ref{03}) and (\ref{08}), an easy computation shows
$$\partial\circ G(v')=\partial (\alpha(v))\,\,\,\,\,\,\,,\,\,\,\,\,G\circ D(v')=G\circ i'(\partial v)=F\circ i'(\partial v)=\alpha(\partial v)$$
$$\partial\circ G(v'')=\partial (\beta(v))=\partial v\,\,\,\,\,\,\,,\,\,\,\,\,G\circ D(v'')=G\circ i''(\partial v)=F\circ i''(\partial v)=\partial v$$
\begin{eqnarray}
\partial \circ G(sv) &=&0\nonumber\\
G\circ D(sv) &=&G\circ (v'-v''-S(\partial v))\,\,\,\,,\,\,\,\,\,\,\,\,\,\,\,\,\,\,\,\,\,\,\,\,\,\,\,\,\,\text{ by (\ref{8}) }\nonumber\\
&=&G (v')-G(v'')-G(S(\partial v))\nonumber\\
&=&\alpha(v)-\beta(v)-G(S(\partial v)))\,\,\,\,,\,\,\,\,\,\,\,\,\,\,\,\,\,\,\,\,\,\,\text{ by (\ref{08}) }\nonumber\\
&=&F\big(S(\partial v)\big)-G(S(\partial v))\nonumber\\
&=&0\nonumber
\end{eqnarray}
Here we use  the fact that  $S(\partial v)\in \T_{q}(V'_{\leq n}\oplus V_{\leq n}''\oplus (sV)_{\leq n})$ and on  $V'_{\leq n}\oplus V_{\leq n}''\oplus (sV)_{\leq n}$,  $G$ and $F$ coincide.

\noindent Finally, it easy to check  (again by using (\ref{8})) that $G\circ i'(v)=\alpha(v)$ and $G\circ i''(v)=\beta(v)$ implying that $[\beta]=[\alpha]$.
\end{proof}

Thus  Lemma \ref{l1} and the relation (\ref{13}) allow us to  define a map
$$\Phi:\ker\,\mathrm{g}\to \mathrm{Hom}\big(V_q, H_{q}(\T(V_{\leq n}))\big)$$
 by setting $\Phi([\beta])(v)=\{\theta_{\beta}(v)\}$ for $v\in V_q$ where $[\beta]$ is chosen as in Lemma \ref{l1}
\begin{proposition}
\label{p1} The map $\Phi$ is an isomorphism.
\end{proposition}
\begin{proof}
Assume that  $\Phi([\beta])(v)=\Phi([\beta'])(v)$ in $H_{q}(\T(V_{\leq n}))$, then $\theta_{\beta'}(v)-\theta_{\beta}(v)=\beta(v)-\beta'(v)$ is a boundary and Lemma \ref{l3} implies that $[\beta]=[\beta']$. Hence $\Phi$ is one to one.

\noindent  Given a homomorphism  $\chi\in \mathrm{Hom}\big(V_q, H_{q}(\T(V_{\leq n}))\big)$ and write $\chi(v)=\{\widetilde{\chi(v)}\}$, where $\widetilde{\chi(v)}$ is a cycle. We define $\beta:(\T (V),\partial)\to (\T (V),\partial)$ by:
$$\beta(v)=v+\widetilde{\chi(v)}\,\,\text{ for }v\in V_q\,\,\,\,\,\,\text{ and }\,\,\beta=id \,\,\text{ on } V_{\leq n}$$
Then $\beta$ is a chain algebra morphism with $\Phi([\beta])=\chi$. Hence $\Phi$ is onto.

\noindent   Finally, given  $\beta,\beta'\in \ker\mathrm{g}$ as in Lemma \ref{l1}. So $\beta(v)=v+\theta_{\beta}(v)$ and  $\beta'(v)=v+\theta_{\beta'}(v)$  for $v\in V_{q}$. Therefore by (\ref{15}) we get: $$\beta'\circ\beta(v)=v+\theta_{\beta'}(v)+\theta_{\beta}(v)=v+\theta_{\beta'\circ\beta}(v)$$
 Consequently $\Phi([\beta'].[\beta])=\Phi([\beta'\circ\beta])=\theta_{\beta'\circ\beta}=\theta_{\beta'}+\theta_{\beta}=\Phi([\beta'])+\Phi([\beta])$. Thus $\Phi$ is a homomorphism of groups.
\end{proof}
\medskip

Summarising, we have proven:
\begin{theorem}\label{t8}
   Let $(\T( V_{q}\oplus V_{\leq n}),\partial)$ be a 1-connected chain algebra. Then there exists
   a short exact sequence of groups:
\begin{equation}\label{34}
\mathrm{Hom}\big(V_q, H_{q}(\T(V_{\leq n}))\big)\rightarrowtail\mathcal{E}(\T(V_{q}\oplus V_{\leq n})) \twoheadrightarrow\D^{q}_{n}
\end{equation}
\end{theorem}

We now focus on the subgroup $\mathcal{E}_*(\T(V_{q}\oplus V_{\leq n}))$ of $\mathcal{E}(\T(V_{q}\oplus V_{\leq n}))$ of the elements inducing the identity on the graded homology  module $H_{*}(V,d)$. Let us define $\G^{q}_{n}$ as the  subgroup of $\E_{*}(\T(V_{\leq n}))$ of the elements $[\alpha]$ satisfying $H_{q-1}(\alpha)\circ b_{q}=b_{q}$ where   $b_{q}:V_q\to H_{q-1}(\T(V_{\leq n}))$ is as in (\ref{37}).
\begin{theorem}
\label{t9}
 Let  $q>n$ and  let $(\T( V_{q}\oplus V_{\leq n}),\partial)$ be a 1-connected chain algebra. Then there exists
   a short exact sequence of groups:
\begin{equation}\label{7}
 \mathrm{Hom}\big(V_q, H_{q}(\T(V_{\leq n}))\big)\rightarrowtail\mathcal{E}_*(\T(V_{q}\oplus V_{\leq n}))
 \twoheadrightarrow  \G^{q}_{n}
\end{equation}
\end{theorem}
\begin{proof}
First let $[\beta]\in\ker \mathrm\,{g} $. Lemma  \ref{l1} assures  that $\widetilde{\alpha}_{q}=id_{V_{q}}$ and $\alpha_{n}= id_{\T(V_{\leq
   n})}$, therefore   $\widetilde{\alpha}=id_{V}$. It follows that $\ker\mathrm\,{g} \subseteq \E_*(\T( V_{q}\oplus V_{\leq n}))$.

  Next from (\ref{34}) we obtain
$$
 \mathrm\,{g}\Big(\E_{*}(\T( V_{q}\oplus V_{\leq n})\Big)=\Big\{\Psi([\alpha])=(\widetilde{\alpha}_{q},[\alpha_{ n}])  \,\,\, \mid \,\, [\alpha]\in \E_{*}(\T( V_{q}\oplus V_{\leq n}) \Big\}.
 $$
  As $[\alpha]\in \E_{*}(\T( V_{q}\oplus V_{\leq n})$,  the graded automorphism  $H_{*}(\widetilde{\alpha})$ is the identity which, in turn, implies    $\widetilde{\alpha}_{q}=\mathrm{id}_{V^{q}}$ and as the pair $(\mathrm{id}_{V^{q}},[\alpha_{ n}])$ makes the diagram (\ref{11}) commutes,   we can identify $\mathrm\,{g}\Big(\E_{\sharp}(\Lambda( V^{q}\oplus V^{\leq n}))\Big)$ with the subgroup $\G^{q}_{n}$.
\end{proof}
\begin{corollary}
\label{c1}
Let $(\T( V_{q}\oplus V_{\leq n}),\partial)$ be a 1-connected chain algebra. If   $\mathcal{E}_*(\T( V_{\leq n}))$ is trivial, then:
\begin{equation}\label{1}
\mathrm{Hom}\big(V_q, H_{q}(\T(V_{\leq n}))\big)\cong\mathcal{E}_*(\T(V_{q}\oplus V_{\leq n}))
\end{equation}
\end{corollary}
\begin{corollary}
	\label{c4}
Let  $(\T( V_{2n}\oplus V_{\leq 2n-1}),\partial)$ be a $n$-connected chain algebra, e.i, $V_{k}=0$ for $k<n$. Then
	\begin{equation}\label{2}
	\mathrm{Hom}\big(V_{2n}, V_{n}\otimes V_{n}\big)\cong\mathcal{E}_*(\T(V_{2n}\oplus V_{\leq 2n-1}))
	\end{equation}
\end{corollary}
	\begin{proof}
	First  as $(\T( V_{2n}\oplus V_{\leq 2n-1}),\partial)$ is  $n$-connected, the group $\mathcal{E}_*(\T( V_{\leq 2n-1}))$ is trivial. Next clearly $H_{2n}(\T(V_{\leq 2n-1})=V_{n}\otimes V_{n}$, hence (\ref{2})   follows from corollary \ref{c1}
\end{proof}
\begin{corollary}\label{c}
	Let $(\T( V_{q}\oplus V_{\leq n}),\partial)$ be a 1-connected chain algebra.  If the group  $\E(\T(V_{q}\oplus V_{\leq n}))$ is finite, then the linear map $b_{q}$ is injective.
\end{corollary}
\begin{proof}
	Assume that  $b_{q}$ is not  injective and let $v_0\neq 0\in V_{q}$ such that $b_{q}(v_0)=0$. For every $a\neq 0\in \Bbb Q$,  we define $\xi_{a}:V_{q}\to V_{q}$ by
	$$\xi_{}(v_0)=av_0\,\,\,\,\,\,\,\,\,\,\,\,\,\,\,\,\,\,\,\,,\,\,\,\,\,\,\,\,\,\,\,\,\,\,\,\,\,\,\xi_a=id\,\,\,\,\,\,\text{otherwise} $$
	Clearly the pair $(\xi_{a},[id])\in\aut(V_{q})\times \mathcal{E}(\T( V_{\leq n}))$ for every $a\neq 0\in \Bbb Q$ and   makes following
	diagram commute
	
	\begin{picture}(300,90)(0,30)
	\put(60,100){$V^{q}\hspace{1mm}\vector(1,0){153}\hspace{1mm}V^{q}$}
	\put(69,76){\scriptsize $b_{q}$} \put(238,76){\scriptsize $b_{q}$}
	\put(66,96){$\vector(0,-1){38}$} \put(235,96){$\vector(0,-1){38}$}
	\put(155,103){\scriptsize $\xi^a$} \put(145,50){\scriptsize
		$id$} \put(35,46){$H_{q-1}(\Lambda V^{\leq n})
		\hspace{1mm}\vector(1,0){110}\hspace{1mm}H_{q-1}(\Lambda V^{\leq n})
		\hspace{1mm}$}
	\end{picture}
	
	\noindent Therefore   $(\xi_{a},[id])\in\D^{q}_{n-1}$ for every $a\neq 0\in \Bbb Q$ implying that the group $\D^{q}_{n}$ is infinite. Consequently the group  $\E(\T(V_{q}\oplus V_{\leq n}))$ is also infinite according the exact sequence  (\ref{34})
\end{proof}

\subsection{$r$-mild differential graded Lie algebras}
Let  $R\subseteq \Bbb Q$ is a ring such that, for some prime $p$, $R$ contains $n^{-1}$ for $n<p$. A free differential graded Lie algebra  $\L(V), \partial )$ over $R$  is called $ r$-mild if
$$V_{k}=0\,\,\,\,\,\,\,\,,\,\,\,\,\,\,\,\,k\leq r-1\,\,\,\,\,\,\,\,,\,\,\,\,\,\,\,\,k\geq pr-1$$
Recall that, in \cite{An1}, Anick defined  a reasonable concept of ''homotopy''
among morphisms between free $R$-dgls, analogous in many respects to the topological notion
of homotopy and proved the following result
\begin{theorem} (\cite{An1} proposition 3.3)\\
	\label{t1}
Let $f, g: (\L(V), \partial ) \to (\L(W), \partial )$  be two $r$-mild free differential graded Lie algebras. Then $f,g$  are homotopic as  dgl-morphisms if and only if $Uf, Ug$ are homotopic as  chain algebra morphisms.
\end{theorem}
\begin{proposition}\label{p2}
Let $\L(V), \partial )$ be an $r$-mild free differential graded Lie algebra. For any $r\leq n<q<rp$,	the homomorphism
	$$\phi_{n}^{q}:\mathcal{E}(\L(V_{q}\oplus V_{\leq n}))\to \mathcal{E}(\T(V_{q}\oplus V_{\leq n}))\,\,\,\,\,,\,\,\,\,\,\,\,\phi_{n}^{q}([\alpha])=[U(\alpha)]$$ 	
is injective.	
\end{proposition}
\begin{proof}
First,  it is well known that if $\alpha$ is equivalence of homotopy, then so is $U(\alpha)$,   hence $\phi_{n}^{q}$ is well defined.

\noindent Next if $U(\alpha)\simeq id_{\T(V_{q}\oplus V_{\leq n})}$, then from  theorem \ref{t1} we deduce that  $\alpha\simeq id_{\L(V_{q}\oplus V_{\leq n})}$. Therefore $\phi_{n}^{q}$ is injective.
 \end{proof}
Using the same argument we can deduce
\begin{corollary}\label{c2}
	Let $\L(V), \partial )$ be an $r$-mild free differential graded Lie algebra. For any $r\leq n<q<rp$,	the homomorphism
$$\psi_{n}^{q}:\mathcal{E}_{*}(\L(V_{q}\oplus V_{\leq n}))\to \mathcal{E}_{*}(\T(V_{q}\oplus V_{\leq n}))\,\,\,\,\,,\,\,\,\,\,\,\,\psi_{n}^{q}([\alpha])=[U(\alpha)]$$
	is  injective.	
\end{corollary}
\section{Topological applications}

Let $X$ be  a simply connected  CW-complex of dimension $n+1$. For  $q>n$  let:
\begin{equation}\label{101}
Y=X\cup_{\alpha}\Big(\underset{i\in I}{\bigcup} e_{i}^{q+1}\Big)
\end{equation}
is the space obtained by attaching cells of dimension $q+1$  to  $X$ by a map $\alpha:\underset{i\in I}{\vee} \Bbb S_{}^{q+1}\to X$.

Recall that the Adams-Hilton  model of $Y$ is  a chain algebra morphism  $$\Theta_{Y}: (\T( V_{q}\oplus V_{\leq n}),\partial)\to C_{*}(\Omega Y,R)$$
such that
$$H_{*}(\Theta_{Y}): H_{*}(\T( V_{q}\oplus V_{\leq n}),\partial)\to H_{*}(\Omega Y,R)$$
is an isomorphism of graded algebras and such as
\begin{equation}\label{0114}
H_{i-1}(V_{q}\oplus V_{\leq n},d)\cong H_{i}(Y,R)\,\,\,\,\,\,\,\,\,\,\,\,,\,\,\,\,\,\,\,\text { as graded modules}
\end{equation}
Here $C_{*}(\Omega Y,R)$ denotes the complex of the non-degenerate cubic chains equipped with the multiplication induced by the composition of loops. We denote by $A(Y)$ the chain algebra $(\T( V_{q}\oplus V_{\leq n}),\partial)$.

Notice also that the  free module $V_{i}$ admit for basis the set of the  cells of dimension $i+1$ of $Y$ and the differential $\partial$ is determined by the attached maps of the cells ( see for example \cite{An1} for more details). Consequently we have
\begin{equation}\label{015}
V_{q}\cong H_{q+1}(Y,X;R).
\end{equation}
where $H_{q+1}(Y,X;R)$ denotes the free  $R$-module of the  homology of the pair $(Y,X)$ in degree $q+1$.

\noindent By the virtues of  the Adams-Hilton model, the  chain algebra $(\T( V_{\leq n}),\partial)$ may be considered as the Adams-Hilton model of    $X$, i.e., $A(X)=(\T( V_{\leq n}),\partial)$. Moreover if $[f]\in \E(A(X))$, then  $f$ induces   the following commutative diagram

\begin{picture}(300,90)(-10,-30)
\put(155,50){\scriptsize
	$H_{q-1}( f)$}
\put(45,-6){$ H_{q-1}(\Omega X,R)
	\hspace{1mm}\vector(1,0){130}\hspace{1mm}H_{q-1}(\Omega X,R)
	\hspace{1mm}$}
\put(45,46){$H_{q-1}(\T(V_{\leq n}))
	\hspace{1mm}\vector(1,0){130}\hspace{1mm}H_{q-1}(\T(V_{\leq n}))
	\hspace{1mm}$}
\put(69,20){\scriptsize $H_{q-1}(\Theta_{X})$} \put(247,20){\scriptsize $H_{q-1}(\Theta_{X})$}
\put(66,42){$\vector(0,-1){38}$} \put(245,42){$\vector(0,-1){38}$}
\put(57,20){\scriptsize $\cong$} \put(236,20){\scriptsize $\cong$}
\put(100,-15){\scriptsize	$H_{q-1}(\Theta_{X}))\circ H_{q-1}( f)\circ (H_{q-1}(\Theta_{X}))^{-1}$}
\end{picture}

\begin{definition}\label{d5}
	\noindent Set $\widetilde{H}_{q-1}(f)=H_{q-1}(\Theta_{X}))\circ H_{q-1}( f)\circ (H_{q-1}(\Theta_{X}))^{-1}$. We define  $\Gamma^{q+1}_{n}$ to be the subset  of $\aut\big(H_{q+1}( Y;R)\big)\times \E(A(X))$ of the pairs $(\xi,[f])$ making the following diagram commutes
\begin{equation}	\label{3}
	\begin{picture}(300,90)(-10,30)
	\put(40,100){$V_{q}\cong H_{q+1}(Y,X;R)\hspace{1mm}\vector(1,0){90}\hspace{1mm}H_{q+1}(Y,X;R)\cong V_{q}$}
	\put(69,76){\scriptsize $b_{q}$} \put(248,76){\scriptsize $b_{q}$}
	\put(66,96){$\vector(0,-1){38}$} \put(245,96){$\vector(0,-1){38}$}
	\put(165,103){\scriptsize $\xi$} \put(155,50){\scriptsize
		$\widetilde{H}_{q-1}(f)$}
	\put(45,46){$ H_{q-1}(\Omega X,R)
		\hspace{1mm}\vector(1,0){130}\hspace{1mm}H_{q-1}(\Omega X,R)
		\hspace{1mm}$}
	\end{picture}
	\end{equation}
	\noindent and
	\begin{equation}\label{22}
	\Pi^{q+1}_{n}=\Big\{[f]\in \mathcal{E}_{*}(A(X))\,\,\,|\,\,\,\widetilde{H}_{q-1}( f)\circ b_{q}=b_{q}\Big\}
	\end{equation}
\end{definition}
Clearly  $\Gamma^{q+1}_{n}$  is a subgroup of $\aut\big(H_{q+1}(Y,X;R)\big)\times \E(A(X))$ and $\Pi^{q+1}_{n}$ is a subgroup of $\mathcal{E}_{*}(A(X))$.
\begin{remark}
	\label{r3}
It is important to notice that if the homomorphism $b_{q}$ is nil, then $$\Gamma^{q+1}_{n}=\aut\big(H_{q+1}(Y,X;R)\big)\times \E(A(X))\,\,\,\,\,\,\,\,\,\,\,\,,\,\,\,\,\,\,\,\,\,\,\,\Pi^{q+1}_{n}= \E_*(A(X))$$
and if $b_{q}$ is an isomorphism, then  from the commutative diagram (\ref{2}) we deduce that $\xi=(b_{q})^{-1}\circ H_{q-1}(f)\circ b_{q}$. Therefore the map
	$$ \E(A(X))\to\Gamma^{q+1}_{n} \,\,\,\,\,\,\,\,\,\,\,\,\,\,\,\,\,,\,\,\,\,\,\,\,\,\,\,\,\,\,\,\,\,[f]\longmapsto\Big((b_q)^{-1}\circ H_{q-1}(f)\circ b_q,[\alpha]\Big)$$
	is an isomorphism.  In this case,  if $[f]\in\Pi^{q+1}_{n}$, then $\widetilde{H}_{q-1}( f)\circ b_{q}=b_{q}$ and as  $b_{q}$ is an isomorphism it follows that $\widetilde{H}_{q-1}( f)=id$. Consequently
	$$\Pi^{q+1}_{n}=\Big\{[f]\in \mathcal{E}_{*}(A(X))\,\,|\,\,\,\widetilde{H}_{q-1}( f)=id\Big\}$$
\end{remark}
\begin{theorem}\label{t6}
	Let $X$ be a   CW-complex simply connected of dimension $n+1$ and let $Y$ as in \rm(\ref{101}).  Then there exist two
	short exact sequences of groups
	\begin{equation}
	\label{181}
	\underset{i}{\oplus}\,H_{q}(\Omega X,R)\rightarrowtail
	\mathcal{E}(A(Y))\overset{}{
		\twoheadrightarrow}\Gamma^{q+1}_{n}\,\,\,\,\,\,\,,\,\,\,\,\,\underset{i}{\oplus}\,H_{q}(\Omega X,R) \rightarrowtail
	\E_{*}(A(Y))\overset{}{
		\twoheadrightarrow}\Pi^{q+1}_{n}
	\end{equation}
 \end{theorem}
\begin{proof}
	The  two sequences (\ref{181})   follow from  a mere transcription of theorems \ref{t8} and \ref{t9} in the topological context by using  the  properties of the Adams-Hilton model. Note that this model implies the   identifications $\Gamma^{q+1}_{n}\cong\D^{q}_{n}$ and $\Pi^{q+1}_{n}\cong\G^{q}_{n}$
\end{proof}
Combining remark \ref{r3} and theorem \ref{t6} we derive the following results
\begin{corollary}\label{c3}
	Let $X$ be a   CW-complex simply connected of dimension $n+1$ and let $Y$ as in \rm(\ref{101}). If the homomorphism  $b_{q}:H_{q+1}(Y,X;R)\to  H_{q-1}(\Omega X,R)$ is bijective, then  there exist two
	short exact sequences of groups
	\begin{equation}
	\label{011}
	\underset{i}{\oplus}\,H_{q}(\Omega X,R)\rightarrowtail
	\mathcal{E}(A(Y))\overset{}{
		\twoheadrightarrow}\mathcal{E}(A(X))
		\end{equation}
	$$\underset{i}{\oplus}\,H_{q}(\Omega X,R) \rightarrowtail
	\E_{*}(A(Y))\overset{}{
		\twoheadrightarrow}\Big\{[f]\in \mathcal{E}_{*}(A(X))\,\,|\,\,\,\widetilde{H}_{q-1}( f)=id\Big\}$$
\end{corollary}
\begin{corollary}\label{c6}
	Let $X$ be a   CW-complex simply connected of dimension $n+1$ and let $Y$ as in \rm(\ref{101}). If the homomorphism  $b_{q}:H_{q+1}(Y,X;R)\to  H_{q-1}(\Omega X,R)$ is nil, then  there exist two
	short exact sequences of groups
	\begin{equation}
	\label{012}
	\underset{i}{\oplus}\,H_{q}(\Omega X,R)\rightarrowtail
	\mathcal{E}(A(Y))\overset{}{
		\twoheadrightarrow}\aut\big(H_{q+1}(Y,X;R)\big)\times \E(A(X))
	\end{equation}
	$$\underset{i}{\oplus}\,H_{q}(\Omega X,R) \rightarrowtail
	\E_{*}(A(Y))\overset{}{
		\twoheadrightarrow} \mathcal{E}_{*}(A(X))$$
\end{corollary}
As a consequence of corollaries \ref{c3} and \ref{c6} we derive
\begin{corollary}\label{c7}
	Let $X$ be a   CW-complex simply connected of dimension $n+1$ and let $Y$ as in \rm{(\ref{101})}.
	
\noindent 	If the homomorphism  $b_{q}$ is nil, then
	$$\frac{\mathcal{E}(A(Y))}{\E_{*}(A(Y))}\cong \aut\big(H_{q+1}(Y,X;R)\big)\times \frac{ \E(A(X))}{\mathcal{E}_{*}(A(X))}$$
	 If $b_{q}$ is an isomorphism,  then
	 	$$\frac{\mathcal{E}(A(Y))}{\E_{*}(A(Y))}\cong\frac{ \E(A(X))}{\Big\{[f]\in \mathcal{E}_{*}(A(X))\,\,|\,\,\,\widetilde{H}_{q-1}( f)=id\Big\}}$$
\end{corollary}
\subsection{Anick  model }
We assume that $R \subseteq \Bbb Q$  is a ring with least non-invertible prime $p>2$.
With $R$ fixed, we take $1 \leq  r < k$ satisfying $k < min(r + 2p-3,rp-1)$. When $R = \Bbb  Q$ we assume $r= 1$ and $k$ is infinite.

\noindent Let
$X$ be  $r$-connected finite CW-complex of dimension $n+1\leq k$.  For  $k\geq q>n+1$  let:
\begin{equation}\label{102}
Y=X\cup_{\alpha}\Big(\underset{i\in I}{\bigcup} e_{i}^{q+1}\Big)
\end{equation}
the space (\ref{101}).  Recall that the Anick  model of $Y$ (see \cite{An1,An2} for more details) is  a free   differential graded Lie algebra   $(\L( V_{q}\oplus V_{\leq n}),\partial)$  over $R$
such that
$$H_{*-1}(\L( V_{q}\oplus V_{\leq n}),\partial)\cong \pi_{*}(Y,)\otimes R\,\,\,\,\,\,\,\,\,\,\,\,,\,\,\,\,\,\,\,H_{*-1}(V_{q}\oplus V_{\leq n},d)\cong H_{*}(Y,R)$$
Moreover by virtues of this model  we deduce
\begin{equation}\label{014}
\mathcal{E}_{*}(\L(V_{q}\oplus V_{\leq n})) \cong \mathcal{E}_{*}(X_{R}) \,\,\,\,\,\,\,\,\,\,\,\,,\,\,\,\,\,\,\,\mathcal{E}(\L(V_{q}\oplus V_{\leq n})) \cong \mathcal{E}(X_{R})
\end{equation}
Here $X_{R}$ denotes the $R$-localised of $X$. From proposition \ref{p2} and corollary \ref{c2} we derive the following result
\begin{theorem}
	Let $Y$  be the space in (\ref{102}). The homomorphisms
	$$ \mathcal{E}(Y_{R}) \to  \mathcal{E}(A(Y_{R})) \,\,\,\,\,\,\,\,\,\,\,\,,\,\,\,\,\,\,\, \mathcal{E}_{*}(Y_{R}) \to  \mathcal{E}_{*}(A(Y_{R}))$$
	are  injective
\end{theorem}

\bibliographystyle{amsplain}

\begin{thebibliography}{10}
\bibitem{An1}  D.J.\ Anick, \emph{Hopf algebras up to homotopy,} J. Amer.
Math. soc (1989), 2(3): 417-452.
\bibitem{An2}  D.J.\ Anick, \emph{An R-local Milnor-Moore Theorem,}
Advances in Math (1989), 77: 116-136.
\bibitem{BB} W.  Barcus and M. G. Barratt,
{\em On the homotopy classification of the extensions of a fixed map,}
Trans. Amer. Math. Soc. 88, (1958) 57-74.
\MR{0097060 (20 \#3540)}

\bibitem{Ba} H. J. Baues, {\em Homotopy type and homology}, Oxford Mathematical Monographs,  Oxford University Press, New York,  1996. \MR{1404516 (97f:55001)}

\bibitem{Benk0}M. Benkhalifa,      \emph{On the group of self-homotopy equivalences of $n$-connected and $(3n+2)$-dimensional CW-Complex},
Topology and its Applications, Volume 233, 1-15,
2018.	


\bibitem{Benk1} M. Benkhalifa,  {\em Postnikov decomposition and the the group of self-equivalences of a rationalized space
}. Homology, Homotopy and Applications, vol.19(1), 2017, pp.209-224


\bibitem{Benk13} M. Benkhalifa,  {\em The  group of  self-equivalences  of  simply connected 4-dimensional CW-complex}, International Electronic Journal of Algebra. Vol. 19, 19-34, 2016


\bibitem{Benk3} M. Benkhalifa,   \textit{Realizability of the group of rational self-homotopy  equivalences
}, Journal of Homotopy and Related Structures, Vol. 5(1) , 361-372, (2011).



\bibitem{Benk4} M. Benkhalifa,
{\em Rational self-homotopy equivalences and Whitehead exact sequence,}
J. Homotopy Relat. Struct. {4}  (2009),  111-121



\bibitem{Benk6} M. Benkhalifa, , {\em Whitehead exact sequence and differential graded free Lie algebra,}   J. Math.  (2004),  {\bf 10}, 987-1005.

\bibitem{Benk11} M. Benkhalifa,  {\em On the homotopy type of a chain algebra,
} Homology, Homotopy and Applications, vol.6 (1), 2004, pp.109-135

\bibitem{Benk7} M. Benkhalifa and S. B. Smith,  {\em The effect of cell attachment on the group of self-equivalences of an
	R-local space}. Journal of Homotopy and Related Structures. {Vol. 4(1)}, (2015), 135-144

\bibitem{BS} R. Bott and H. Samelson  \emph{On the Pontryagin product in spaces of paths},  Commentarii Mathematici Helvetici, Vol.27 (1), 320-337, (1953)

\bibitem{Kahn} P. J. Kahn, {\em Self-equivalences of {$(n-1)$}-connected {$2n$}-manifolds,} Bull. Amer. Math. Soc. 72 (1966), 562--566.


\bibitem{OSS} S. Oka,  N.  Sawashita  and M. Sugawara, {\em On the group of self-equivalences of a mapping cone,}  Hiroshima Math. J.  4, (1974), 9-28.

\bibitem{Ru1} J. Rutter, {\em The group of homotopy self-equivalence classes of CW-complexes,} Math. Proc. Cambridge Philos. Soc. 93, (1983), 275-293.

\bibitem{Ru} J. Rutter, {\em Spaces of homotopy self-equivalences: A survey,} Lecture Notes in Mathematics,  162,  Springer-Verlag, Berlin (1997).
\end{thebibliography}

\end{document}